\documentclass[11pt]{article}

\usepackage{amsmath,amssymb,amsfonts,amsthm}

\usepackage{tikz}
\usetikzlibrary{automata}
\usetikzlibrary{arrows}
\usetikzlibrary{positioning,calc}
\usetikzlibrary{graphs}
\usetikzlibrary{graphs.standard}
\usetikzlibrary{arrows,decorations.markings}
\usepackage{tkz-graph}
\usetikzlibrary{chains,fit,shapes}
\usetikzlibrary{calc}

\tikzset{every loop/.style={min distance=10mm,looseness=10}}
\tikzset{every state/.style={minimum size=2mm}}

\newtheorem{theorem}{Theorem}
\newtheorem{example}{Example}
\newtheorem{lemma}[theorem]{Lemma}
\newtheorem{problem}{Problem}
\newtheorem{remark}[theorem]{Remark}

\newtheorem{Fact}[theorem]{Observation}

\title{On semi-transitive orientability of split graphs}

\author{Sergey Kitaev\footnote{Department of Mathematics and Statistics, University of Strathclyde, 26 Richmond Street, Glasgow G1, 1XH, United Kingdom. 
{\bf Email:} sergey.kitaev@strath.ac.uk.}\ \ and Artem Pyatkin\footnote{Sobolev Institute of Mathematics, Koptyug ave, 4, Novosibirsk, 630090, Russia}\ \footnote{Novosibirsk State University, Pirogova str. 2, Novosibirsk, 630090, Russia. {\bf Email:} artem@math.nsc.ru.}}

\begin{document}

\maketitle

\begin{abstract}
A directed graph is semi-transitive if and only if it is acyclic and for any directed path $u_1\rightarrow u_2\rightarrow \cdots \rightarrow u_t$, $t \geq 2$, either there is no edge from $u_1$ to $u_t$ or all edges $u_i\rightarrow u_j$ exist for $1 \leq i < j \leq t$. Recognizing semi-transitive orientability of a graph is an NP-complete problem. 

A split graph is a graph in which the vertices can be partitioned into a clique and an independent set. Semi-transitive orientability of spit graphs was recently studied in the literature. The main result in this paper is proving that recognition of semi-transitive orientability of split graphs can be done in a polynomial time. We also characterize, in terms of minimal forbidden induced subgraphs, semi-transitively orientable split graphs with the size of the independent set at most 3, hence extending the known classification of such graphs with the size of the clique at most 5. \\

\end{abstract}

\section{Introduction}
A directed graph is semi-transitive if and only if it is acyclic and for any directed path $u_1\rightarrow u_2\rightarrow \cdots \rightarrow u_t$, $t \geq 2$, either there is no edge from $u_1$ to $u_t$ or all edges $u_i\rightarrow u_j$ exist for $1 \leq i < j \leq t$. 
An undirected graph is {\em semi-transitive} if it admits a semi-transitive orientation. 
The notion of a semi-transitive orientation was introduced in \cite{HKP16} to characterize {\em word-representable graphs} that were introduced in \cite{KP08}. These graphs generalize several important and well-studied classes of graphs; see a survey \cite{KP18} or a book \cite{KL15} for more details.

A graph $G=(V,E)$ is called a {\it split graph} if its vertex set $V$ can be partitioned into two parts $I\cup C$ such that $C$ induces a clique, and $I$ induces an independent set \cite{FH77}. We assume that a clique in a split graph is maximum possible, i.e.\ none of the vertices from $I$ is adjacent to all vertices of $C$. The study of split graphs attracted much attention in the literature (e.g.\ see \cite{CJKS2020} and references therein). Some split graphs are semi-transitive, others are not. Semi-transitive orientability (equivalently, word-representability) of split graphs was recently studied in the papers \cite{CKS21,I2021,IK2021,KLMW17}. Interestingly, split graphs were instrumental in \cite{CKS21} to solve a 10 year old open problem in the theory of word-representable graphs.

In general, recognizing if a given graph is semi-transitive is an NP-complete problem \cite{KL15}. The main result in this paper is showing that for the class of split graphs this problem is  polynomially solvable. Thus, our main focus is the following problem. 

\begin{problem}~\label{P1}
Given a split graph $G$, is it semi-transitive?
\end{problem}

Problem~\ref{P1} was studied in \cite{CKS21,I2021,KLMW17}. In particular, a characterization of semi-transitive split graphs in terms of minimal forbidden induced subgraphs for $|C|\le 5$ follows from \cite{CKS21,KLMW17}. Moreover, in \cite{KLMW17} the following structural property of semi-transitive split graphs was proved, where $N(v)$ denotes the neighbourhood of a vertex $v$. Clearly, $N(v)\subset C$ for each $v\in I$. Also, for any two integers $a\le b$, we let $[a,b]=\{a,a+1,\ldots,b\}$.

\begin{theorem}[\cite{KLMW17}]~\label{t1}
A split graph $G$ is semi-transitive if and only if the vertices of $C$ can be labeled from $1$ to $k=|C|$ in such a way that:
\begin{itemize}
\item[{\em (1)}] For each $v\in I$, either $N(v)=[a,b]$ for $a\le b$ or $N(v)=[1,a]\cup [b,k]$ for $a< b$.
\item[{\em (2)}] If $N(u)=[a_1,b_1]$ and $N(v)=[1,a_2]\cup [b_2,k]$, for $a_1\le b_1,\ a_2<b_2$, then $a_1>a_2$ or $b_1<b_2$.
\item[{\em (3)}] If $N(u)=[1,a_1]\cup [b_1,k]$ and $N(v)=[1,a_2]\cup [b_2,k]$, for $a_1<b_1$ and $a_2<b_2$, then $a_2<b_1$ and $a_1<b_2$.
\end{itemize}
\end{theorem}

Note that Theorem~\ref{t1} is a convenient, for our purposes, reformulation of Theorem~15 in \cite{KLMW17}. Indeed, once a proper labeling of a split graph is done, to create a semi-transitive orientation of the graph, there is a unique way to  orient edges in $C$ and between $v\in I$ and $N(v)=[1,a]\cup [b,k]$ (if there are any such vertices), and there are two choices to orient edges between  $v\in I$ and $N(v)=[a,b]$ (if there are any such vertices), namely, $v$ can be a source or a sink.

Determining semi-transitivity of split graphs is closely related to the well-known circular ones property of $(0,1)$-matrices. A $(0,1)$-matrix has the {\it consecutive ones} property (for columns) if after some permutation of its rows in all columns the ones are consecutive. A $(0,1)$-matrix has the {\it circular ones} property (for columns) if after some permutation of its rows in all columns either ones or zeroes are consecutive. Note that if zeroes are consecutive then ones are  ``almost consecutive'' in the sense that they are allowed to wrap around from the bottom of a column to its top. 

Note that from the algorithmic point of view, searching for permutations of rows giving a consecutive ones property and giving a circular ones property
is equivalent, as follows from the following lemma:


\begin{lemma}[\cite{T71}]~\label{l1}
Let $M$ be a $(0,1)$-matrix. Denote by $M_1$ the matrix obtained from $M$ by the inversion (changing $0$s by $1$s and vice versa) of all columns having $1$ in the first row. Then $M$ has the circular ones property if and only if $M_1$ has the consecutive ones property.
\end{lemma}

As for the consecutive ones property, the first polynomial algorithm for determining whether a $(0,1)$-matrix has it was suggested in \cite{FG65}. In \cite{BL76} a more general algorithm based on the concept of PQ-trees was presented (in particular, it allows to find all possible permutations of rows providing consecutive ones orderings for all columns); it solves the decision problem in time $\mathcal{O}(m+n+f)$ where $m,n,$ and $f$ are the number of rows, columns and ones in $M$, respectively. Moreover, it was shown there that the problem of finding out whether a $(0,1)$-matrix has a circular ones property can be solved in the same time $\mathcal{O}(m+n+f)$. In \cite{HM03} an algorithm finding all possible permutations providing circular ones orderings for all columns in linear time 
is suggested. In \cite{CLMNSSS} a linear time algorithm for checking the isomorphism of any two $(0,1)$-matrices having circular one property was constructed; this solves the graph isomorphism problem in linear time for several graph classes.

Given a split graph $G$ with $C=\{ u_1,\ldots,u_k\}$ and $I=\{v_1,\ldots,v_t\}$, consider a $(0,1)$-matrix $M(G)$ with $k$ rows and $t$ columns where $m_{ij}=1$ if and only if $u_i$ is adjacent to $v_j$. Then, clearly, any labelling of the vertices in $C$  defines a permutation of the rows of the matrix $M(G)$. Moreover, such a labelling satisfies condition (1) of Theorem~\ref{t1} if and only if the corresponding permutation provides a circular ones property. The first attempts of translating conditions (2)--(3) of Theorem~\ref{t1} into the matrices language were made in \cite{I2021,IK2021}, but the condition stated there was erroneous. The correct statement is as follows: 

\begin{theorem}~\label{t1m}
A split graph $G$ is semi-transitive if and only if the rows of matrix $M(G)$ can be permuted in such a way that:
\begin{itemize}
\item[{\em (i)}] The ordering has a circular ones property for all columns.
\item[{\em (ii)}] If a column has the form $1^a0^b1^c$ where $a+b+c=k$ and $a,b,c\ge 1$ then no other column may contain ones in all positions from $a$ to $a+b+1$.
\end{itemize}
\end{theorem}

It is easy to verify that property (ii) of Theorem~\ref{t1m} is equivalent to conditions (2)--(3) of Theorem~\ref{t1}. 

So, a natural attempt to solve Problem~\ref{P1} would be as follows: consider all permutations of rows providing circular ones ordering using the result from \cite{HM03} and check whether they satisfy property (ii) of Theorem~\ref{t1m}. Unfortunately, this approach does not provide a polynomial-time algorithm for Problem~\ref{P1}. Indeed, although the result from \cite{HM03} allows to find each permutation with circular ones property in linear time, the total number of such permutations may be exponential, while the number of permutations satisfying  property (ii) of Theorem~\ref{t1m} among them could be small.

\begin{example} Consider the following matrix $M$ with $k$ rows and $k+2$ columns.

\begin{equation*}
M = \left(
\begin{array}{ccccccc}
1 & 0 & \ldots & 0 & 0 & 1 & 0\\
0 & 1 & \ldots & 0 & 0 & 1 & 1\\
\vdots & \vdots &\ddots &\vdots & \vdots & \vdots & \vdots\\
0 & 0 & \ldots & 1 & 0 & 1 & 1\\
0 & 0 & \ldots & 0 & 1 & 0 & 1
\end{array}
\right)
\end{equation*}
Since each column contains either one $1$, or one $0$, every permutation of its rows has circular ones property, i.e. there are $k!$ such permutations. On the other hand, only those permutations where $0$'s in the last two columns are in consecutive rows satisfy property (ii) of Theorem~\ref{t1m} (rows $1$ and $k$ are also considered as consecutive here); hence, there are only $2k$ such permutations. \end{example}

In order to apply the circular ones permutations techniques for solving Problem~\ref{P1} we need a deeper study of the structure of semi-transitive split graphs provided in the following section.

The paper is organized as follows. In Section~\ref{sec2} we prove that Problem~\ref{P1} is equivalent to an auxiliary problem of a subsets bijection into a cycle, and show that this problem can be reduced to a known polynomially solvable problem of determining $(0,1)$-matrices with the circular ones property. 
In Section~\ref{sec3} we use a derived by us equivalence to characterize, in terms of minimal forbidden induced subgraphs, all semi-transitive split graphs with $|I|\le 3$.

\section{Polynomial solvability of Problem~\ref{P1}}\label{sec2}
Denote by $V(C_k)$ the vertex set of a cycle graph $C_k$. 
Consider the following auxiliary problem.

\begin{problem}~\label{P2}
Given a set $A$ of cardinality $k$ and its proper subsets $A_1,\ldots,A_t\subset A$, is there a bijection $F:A\longrightarrow V(C_k)$ such that
for all $i,j\in [1,t]$ the subgraph induced by $F(A_i\cap A_j)$ is connected (the empty subgraph is considered to be connected)?
\end{problem}

Note that in Problem~\ref{P2}, the subgraphs induced by each $F(A_i)$ must be connected (this corresponds to the case when $j=i$). The following theorem shows the equivalence of Problems~\ref{P1} and~\ref{P2}.

\begin{theorem}~\label{t2}
Problem~\ref{P1} is polynomially solvable if and only if Problem~\ref{P2} is polynomially solvable.
\end{theorem}

\begin{proof}
Assume that Problem~\ref{P2} is polynomially solvable. Consider an arbitrary instance of Problem~\ref{P1} with the graph $G=(I\cup C, E)$ where $I=\{v_1,\ldots,v_t\},\ C=\{u_1,\ldots,u_k\}$. Construct a corresponding instance of Problem~\ref{P2} as follows. Let $A=C$ and put $A_i=N(v_i)$ for all $i=1,\ldots, t$. Let us prove the following claim.\\[-3mm]

\noindent
{\bf Claim.} Graph $G$ has a semi-transitive orientation if and only if the set $A$ has an appropriate bijection.\\[-3mm]

Indeed, if $G$ satisfies the conditions of Theorem~\ref{t1} then let the labeling of the vertices in $C$ define the bijection $F$ (the order of the vertices in the cycle graph $C_k$). Then (1) ensures that each subgraph induced by $F(A_i)$ is connected since it is a path from $a$ to $b$ that either goes through $1$ or not. Consider arbitrary $v_i,v_j\in C$. If $N(v_i)=[a_1,b_1]$ and $N(v_j)=[a_2,b_2]$ then $F(A_i\cap A_j)$ induces either an empty subgraph or a path from $\max\{ a_1,a_2\}$ to $\min\{ b_1,b_2\}$; anyway, it is connected. Let $N(v_i)=[a_1,b_1]$ and $N(v_j)=[1,a_2]\cup[b_2,k]$. By (2), either $a_1>a_2$ or $b_1<b_2$ or both. Then, respectively, $F(A_i\cap A_j)$ induces either a path from $b_2$ to $b_1$, or a path from  $a_1$ to $a_2$, or the empty subgraph. Finally, let 
$N(v_i)=[1,a_1]\cup[b_1,k]$ and $N(v_j)=[1,a_2]\cup[b_2,k]$. Then by (3), $a_2<b_1$ and $a_1<b_2$. Hence, $F(A_i\cap A_j)$ induces a path $b,b+1,\ldots ,k,1,2,\ldots, a$  where $a=\min\{ a_1,a_2\}$ and $b=\max\{b_1,b_2 \}$. So, the subgraph induced by $F(A_i\cap A_j)$ is connected in $C_k$ for all $i,j\in [1,t]$.

Assume now that there is a bijection from $A$ into $C_k$ such that all intersections of the subsets induce connected subgraphs. Label $C$ in the order of the cycle graph $C_k$ starting from an arbitrary vertex. Since each subgraph induced by $F(A_i)$ is connected, the corresponding $N(v_i)$ is either an interval $[a,b]$ or the union of two intervals $[1,a]\cup [b,k]$ for some $a<b$, i.e. (1) takes place. If $N(v_i)=[a_1,b_1]$ and $N(v_j)=[1,a_2]\cup [b_2,k]$ but $a_1\le a_2$ and $b_1\ge b_2$ then $F(A_i\cap A_j)$ induces a union of two paths $a_1,\ldots ,a_2$ and $b_1,\ldots, b_2$ which is disconnected, a contradiction. So, (2) holds. Assume that 
$N(v_i)=[1,a_1]\cup[b_1,k]$ and $N(v_j)=[1,a_2]\cup[b_2,k]$. If $a_2\ge b_1$ then $F(A_i\cap A_j)$ induces the union of three paths $1,\ldots ,a_1;\ b_1,\ldots, a_2;$ and $b_2\ldots, k$,  while if
$a_1\ge b_2$ then $F(A_i\cap A_j)$ induces the union of paths $1,\ldots, a_2;\ b_2,\ldots, a_1;$ and $b_1,\ldots, k$. In both cases, the subgraph is clearly disconnected; because of the contradiction, (3) must be true.

By the claim, applying an algorithm solving the constructed instance of Problem~\ref{P2} gives a solution of the initial instance of Problem~\ref{P1}. 

If Problem~\ref{P1} is polynomially solvable then the reduction from an instance of  Problem~\ref{P2} to  Problem~\ref{P1} is made in a similar way: put $C=A, I=\{v_1,\ldots,v_t\},$ and $N(v_i)=A_i$. Then the same claim provides the polynomial solvability of Problem~\ref{P2}.
\end{proof}

In order to solve  Problem~\ref{P2}, we use the above-mentioned results \cite{BL76,HM03} on the polynomial solvability of checking the circular ones property of $(0,1)$-matrices.  Clearly, $\mathcal{O}(m+n+f)$ can be bounded from above by $\mathcal{O}(mn)$, which is the input size of the consecutive ones problem in general. Note that a split graph is defined by the neighbourhood sets of the vertices in $I$. Hence, the size of the input of Problem~\ref{P1} is $\mathcal{O}(tk)$. We can now prove the main result of the paper. 

\begin{theorem}~\label{t3}
Problem~\ref{P1} can be solved in time $\mathcal{O}(t^2k)$.
\end{theorem}

\begin{proof} Given an instance of Problem~\ref{P1}, first construct an equivalent instance of Problem~\ref{P2} as shown in the proof of Theorem~\ref{t2}. Clearly, it takes time $\mathcal{O}(tk)$. Let $A=\{a_1,\ldots,a_k\}$.  Put $m=k$ and $n=(t^2+t)/2$.  Construct the $(0,1)$-matrix $M$ of size $m\times n$ as follows. 
For convenience, let the columns be indexed by the pairs $(i,j)$ where $i,j\in [1,t]$ and $i\le j$. Let each row of $M$ correspond to an element of the set $A$ and each column $(i,j)$ be a characteristic vector of the subset $A_i\cap A_j$ (i.e. $m_{s,(i,j)}=1$ if and only if $a_s \in A_i\cap A_j$). Then, clearly, finding the appropriate bijection $F$ is equivalent to finding a permutation of the rows that provides a circular ones ordering for all columns of the matrix $M$. By the result from \cite{BL76} mentioned above, the latter can be done in time  $\mathcal{O}(mn)=\mathcal{O}(t^2k)$. 
\end{proof}

\begin{remark} Note that in practice many columns can contain no, or just one, $1$; such columns can be omitted since they have the circular ones property after any permutation of rows. This can make the algorithm faster, although in the worst case, we have the bound of $\mathcal{O}(t^2k)$. \end{remark}

\section{Split graphs with small independent set}\label{sec3}
In \cite{CKS21,KLMW17} a characterization, in terms of minimal forbidden induced subgraphs, of all semi-transitive split graphs with $|C|\le 5$ was found. It follows easily from Theorem~\ref{t1} that all split graphs with $|I|\le 2$ are semi-transitive. In this section, we provide a complete characterization, in terms of minimal forbidden induced subgraphs, of semi-transitive split graphs with $|I|=3$. We start with the following easy observation that has been frequently used in the literature~\cite{KL15}.

\begin{Fact}~\label{o1}
Given an arbitrary graph $G=(V,E)$, let $a,b\in V$ satisfy $N(a)\setminus \{b\} = N(b)\setminus \{a\}$. Then $G$ is semi-transitive if and only if $G\setminus b$ is semi-transitive.
\end{Fact}

Indeed, having a semi-transitive orientation of $G\setminus b$, direct each edge $bx$ (for $x\ne a$) in the same way as $ax$; if the edge $ab$ exists, direct it in an arbitrary way. Clearly, the obtained orientation of $G$ must be semi-transitive.

Let $G$ be a split graph with $I=\{ a,b,c\}$. Then, by Observation~\ref{o1} we can assume that for each $S\subseteq I$ there is at most one $v_S\in C$ whose neighbourhood in $I$ coincides with $S$. So, we can assume that $|C|\le 8$.

\begin{theorem}~\label{t4}
Let $G$ be a split graph with $I=\{ a,b,c\}$. Then $G$ is not semi-transitive if and only if
\begin{itemize}
\item[{\em (a)}] $C$ contains $v_\emptyset, v_{\{a,b\}}, v_{\{a,c\}}, v_{\{b,c\}}$; or

\item[{\em (b)}] $C$ contains $v_{\{a,b,c\}}, v_{\{a,b\}}, v_{\{a,c\}}, v_{\{b,c\}}$; or

\item[{\em (c)}] $C$ contains $v_{\{a,b,c\}}, v_{\{a\}}, v_{\{b\}}, v_{\{c\}}$.
\end{itemize}

\noindent
That is, $G$ is not semi-transitive if and only if it contains one of the graphs in Figure~\ref{forbidden-subgraphs} as an induced subgraph.
\end{theorem}

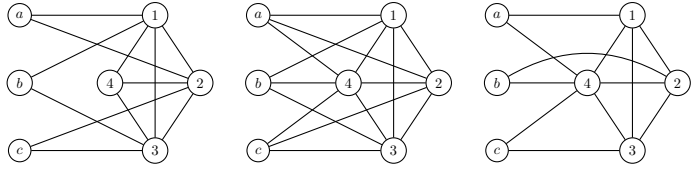
\begin{figure}
\begin{tabular}{c c c}
\hspace{1cm}
\begin{tikzpicture}[scale=0.3]

\draw (6,6) node [scale=0.5, circle, draw](node1){1};
\draw (8,3) node [scale=0.5, circle, draw](node2){2};
\draw (6,0) node [scale=0.5, circle, draw](node3){3};
\draw (4,3) node [scale=0.5, circle, draw](node4){4};

\draw (0,6) node [scale=0.5, circle, draw](node5){$a$};
\draw (0,3) node [scale=0.5, circle, draw](node6){$b$};
\draw (0,0)node [scale=0.5, circle, draw](node7){$c$};

\draw (node1)--(node2)--(node3)--(node4)--(node1);
\draw (node1)--(node3);
\draw (node2)--(node4);

\draw (node1)--(node5);
\draw (node1)--(node6);
\draw (node3)--(node6);
\draw (node3)--(node7);
\draw (node2)--(node5);
\draw (node2)--(node7);

\end{tikzpicture}

& 

\begin{tikzpicture}[scale=0.3]

\draw (6,6) node [scale=0.5, circle, draw](node1){1};
\draw (8,3) node [scale=0.5, circle, draw](node2){2};
\draw (6,0) node [scale=0.5, circle, draw](node3){3};
\draw (4,3) node [scale=0.5, circle, draw](node4){4};

\draw (0,6) node [scale=0.5, circle, draw](node5){$a$};
\draw (0,3) node [scale=0.5, circle, draw](node6){$b$};
\draw (0,0)node [scale=0.5, circle, draw](node7){$c$};

\draw (node1)--(node2)--(node3)--(node4)--(node1);
\draw (node1)--(node3);
\draw (node2)--(node4);

\draw (node1)--(node5);
\draw (node1)--(node6);
\draw (node3)--(node6);
\draw (node3)--(node7);
\draw (node2)--(node5);
\draw (node2)--(node7);
\draw (node4)--(node7);
\draw (node4)--(node5);
\draw (node4)--(node6);

\end{tikzpicture}

&

\begin{tikzpicture}[scale=0.3]

\draw (6,6) node [scale=0.5, circle, draw](node1){1};
\draw (8,3) node [scale=0.5, circle, draw](node2){2};
\draw (6,0) node [scale=0.5, circle, draw](node3){3};
\draw (4,3) node [scale=0.5, circle, draw](node4){4};

\draw (0,6) node [scale=0.5, circle, draw](node5){$a$};
\draw (0,3) node [scale=0.5, circle, draw](node6){$b$};
\draw (0,0)node [scale=0.5, circle, draw](node7){$c$};

\draw (node1)--(node2)--(node3)--(node4)--(node1);
\draw (node1)--(node3);
\draw (node2)--(node4);

\draw (node1)--(node5);
\draw (node3)--(node7);
\draw [bend right=30] (node2) to (node6);
\draw (node4)--(node7);
\draw (node4)--(node5);
\draw (node4)--(node6);

\end{tikzpicture}

\end{tabular}

\caption{Minimal forbidden induced subgraphs for semi-transitive orientability of split graphs with $|I|=3$. Note that these are three out of the four minimal forbidden induced subgraphs used in characterisation of semi-transitive split graphs with $|C|=4$ in \cite{KLMW17}.}\label{forbidden-subgraphs}
\end{figure}

\begin{proof} In each of the cases (a)--(c) we use the reduction to an equivalent instance of Problem~\ref{P2} used in the proof of Theorem~\ref{t2}, and prove the non-existence of a required bijection $F$. \\[-3mm]

\noindent
(a) Since the semi-transitivity property is hereditary (when removing vertices), it is sufficient to consider the case of $C=\{v_\emptyset, v_{\{a,b\}}, v_{\{a,c\}}, v_{\{b,c\}} \}$. Then, in an equivalent instance of Problem~\ref{P2} we have $A_a=\{ v_{\{a,b\}}, v_{\{a,c\}}\}$, $A_b=\{ v_{\{a,b\}}, v_{\{b,c\}}\}$ and $A_c=\{ v_{\{a,c\}}, v_{\{b,c\}}\}.$  Assume that the required bijection from $A$ into the cycle graph $C_4$ exists. Then, the vertex $F(v_\emptyset)$ is adjacent, say, to the vertices $F(v_{\{a,b\}})$ and $F(v_{\{a,c\}})$. But then the subgraph induced by $F(A_a)$ is disconnected. Two other possible choices for the neighbours of $F(v_\emptyset)$ are similar. \\[-3mm]

\noindent
(b) Let $C=\{ v_{\{a,b,c\}}, v_{\{a,b\}}, v_{\{a,c\}}, v_{\{b,c\}} \}$. Here $A_a=\{ v_{\{a,b\}}, v_{\{a,c\}}, v_{\{a,b,c\}}\}$, $A_b=\{ v_{\{a,b\}}, v_{\{b,c\}}, v_{\{a,b,c\}}\}$ and $A_c=\{ v_{\{a,c\}}, v_{\{b,c\}}, v_{\{a,b,c\}}\}.$ Due to the symmetry, we may assume that $F(v_{\{a,b,c\}})$ is adjacent to $F(v_{\{a,b\}})$ and $F(v_{\{a,c\}})$. In this case, the subgraph induced by $F(A_b\cap A_c)=\{ v_{\{b,c\}}, v_{\{a,b,c\}}\}$ is disconnected. 
\\[-3mm]

\noindent
(c) Let $C=\{ v_{\{a,b,c\}}, v_{\{a\}}, v_{\{b\}}, v_{\{c\}}\}$. Then $A_a=\{ v_{\{a\}}, v_{\{a,b,c\}}\},\ A_b=\{ v_{\{b\}}, v_{\{a,b,c\}}\}$ and $A_c=\{ v_{\{c\}}, v_{\{a,b,c\}}\}.$ Again, let 
$F(v_{\{a,b,c\}})$ be adjacent to $F(v_{\{a\}})$ and $F(v_{\{b\}})$. Then, clearly, the subgraph induced by $F(A_c)$ is disconnected. \\[-3mm]

Now assume that none of (a)--(c) takes place. If none of $v_{\{a,b,c\}}, v_\emptyset$ is in $C$ then label the remaining vertices in the order 
$$v_{\{a\}} < v_{\{a,b\}} < v_{\{b\}} < v_{\{b,c\}} < v_{\{c\}} < v_{\{a,c\}}.$$ It is easy to see that the conditions of Theorem~\ref{t1} hold. Assume that $C$ does not contain $v_{\{a,b,c\}}$ but contains $v_\emptyset$. Since (a) is not true, $C$ does not contain, say, $v_{\{a,b\}}$ (two other possibilities are similar). Then label $C$ in the order $$v_{\{a\}} < v_{\{a,c\}} < v_{\{c\}} < v_{\{b,c\}} < v_{\{b\}} < v_\emptyset.$$ Again, by Theroem~\ref{t1}, the semi-transitive orientation exists. Now, assume that $v_{\{a,b,c\}}$ is in $C$. We can further assume that $v_{\{a,b\}}$ is not in $C$ due to (b), and also that one of 
$v_{\{a\}}, v_{\{b\}}, v_{\{c\}}$ is not in $C$ due to (c). If $v_{\{a\}}\not\in C$ then label $C$ in the order  $$v_{\{c\}}< v_{\{a,c\}}< v_{\{a,b,c\}}< v_{\{b,c\}}< v_{\{b\}}< v_\emptyset.$$ The case $v_{\{b\}}\not\in C$ is similar. If  $v_{\{c\}}\not\in C$ then let the labeling of $C$ be in the order $$v_{\{a\}}< v_{\{a,c\}}< v_{\{c,b,c\}}< v_{\{b,c\}}< v_{\{b\}}< v_\emptyset.$$ In either case, the conditions of Theorem~\ref{t1} are satisfied.
\end{proof}

\section{Open problems}

We would like to conclude our paper with the following open problems.

\begin{itemize}
\item[1.] Does the problem of recognition of semi-transitivity remain polynomially solvable for {\em chordal graphs} (a natural superclass of split graphs)? A chordal graph is a graph in which all cycles of four or more vertices have a chord, which is an edge that is not part of the cycle but connects two vertices of the cycle.
\item[2.] Characterize, in terms of minimal forbidden subgraphs or by other means, semi-transitive split graphs with small $|I|\geq4$.  
\item[3.] The characterization of semi-transitivity of split graphs in the case of $|I|=3$ involves three forbidden subgraphs, while such characterization in the cases of $|C|=4$ and $|C|=5$ involves four and nine forbidden subgraphs, respectively (see \cite{CKS21,KLMW17}). Is there any relation between the number of forbidden subgraphs, say, for a fixed $|I|$ and $|C|=|I|+1$? 
\end{itemize}

\section*{Aknowledgements} The work of the second author was partially supported by the program of fundamental scientific researches of
the SB RAS, project 0314-2019-0014.

\end{document}